\documentclass[11pt,reqno]{amsart}
\usepackage{amsmath,amsthm,amssymb,latexsym,color}
\usepackage{hyperref}
\usepackage{lipsum}
\usepackage{xcolor}
\date{}

\usepackage{graphicx}
\newtheorem{theorem}{Theorem}[section]
\newtheorem*{theorem*}{Theorem}

\newtheorem{lemma}[theorem]{Lemma}
\newtheorem{corollary}[theorem]{Corollary}

\newtheorem{proposition}[theorem]{Proposition}

\theoremstyle{definition}

\theoremstyle{plain}
\newcommand{\dd}{\mathrm d}

\usepackage{tikz}

\newcommand\Matching[2]{%
	\begin{tikzpicture}[scale=0.33,font=\footnotesize,baseline=-1mm]
		\draw(-0.5,0) -- ++ (#1,0);
		\foreach \x in {1,...,#1}{
			\draw[circle,fill] (\x-1,0)circle[radius=1mm]node[below]{$\x$};
		}
		\foreach \x/\y in {#2} {
			\draw(\x-1,0) to[bend left=45] (\y-1,0);
		}
	\end{tikzpicture}%
}

 \usepackage{graphicx}

\newcommand{\proj}{\operatorname{proj}}
\newcommand{\tr}{\tau}
\newcommand{\Le}{\mathrm L}
\newcommand{\NC}{M}
\newcommand{\CC}{\mathbb{C}}
\newcommand{\NN}{\mathbb{N}}
\newcommand{\std}{\operatorname{\sigma}}
\newcommand{\cov}{{\rm cov}}
\newcommand{\es}{\emptyset}
\newcommand{\RR}{\mathbb{R}}
\usepackage{mathtools}
\newcommand\ceq{\coloneqq}

\let\originalleft\left
\let\originalright\right
\renewcommand{\left}{\mathopen{}\mathclose\bgroup\originalleft}
\renewcommand{\right}{\aftergroup\egroup\originalright}

\begin{document}

\author{Benjamin Dadoun}
\author{Pierre Youssef}
\title{Maximal correlation and monotonicity of free entropy and Stein discrepancy}

\maketitle

\begin{abstract}\noindent%
  We introduce the
  maximal correlation coefficient~$R(\NC_1,\NC_2)$
  between two noncommutative probability
  subspaces~$\NC_1$ and~$\NC_2$ and show that the maximal
  correlation coefficient between the sub-algebras generated by~$s_n\coloneqq x_1+\ldots +x_n$ and~$s_m:=x_1+\ldots +x_m$ equals
  $\sqrt{m/n}$ for $m\le n$, where $(x_i)_{i\in \NN}$ is a sequence of free and identically distributed noncommutative
  random variables. This is the free-probability
  analogue of a result by Dembo--Kagan--Shepp in classical
  probability. As an application, we use this estimate to provide another simple proof
  of the monotonicity of the free entropy and free Fisher information in the free central limit theorem. 
  Moreover, we prove that the free Stein Discrepancy introduced by Fathi and Nelson is non-increasing along 
  the free central limit theorem. 
  \end{abstract}

\section{Introduction and main result}
Pearson's correlation coefficient
$\rho(X_1,X_2)\ceq\cov(X_1,X_2)/{\std(X_1)\std(X_2)}$
is a standard measure of (bivariate)
dependency. The \textit{maximal correlation} $R(X_1,X_2)$ between
two random variables~$X_1$ and~$X_2$ is then naturally defined as
the supremum of $\rho(f(X_1),g(X_2))$ over $\Le^2$-functions
$f,g$.

In 2001, Dembo, Kagan, and Shepp~\cite{Dembo01}
investigated this maximal correlation for the partial sums
$S_n\ceq X_1+\cdots+X_n$
of i.i.d.\
random variables~$X_i$. Namely, regardless of the $X_i$'s distribution, they
showed that
\[R(S_n,S_m)\le\sqrt{\frac mn},\quad m\le n,\]
with equality if $\sigma(X_1)<\infty$.

The aim of this note is to prove the analogous statement in
the context of free probability, a theory which was initiated by Voiculescu (see \cite{Voiculescu94}) in the 1980's and has since flourished into an established area with 
connections to several other fields. We refer to \cite{Hiai00,Mingo17} for an introduction on the subject and an extensive list of references. 
Before stating our main result, let us first recall the necessary framework. 

Let $(\NC,\tr)$ be a noncommutative, faithful, tracial
probability space, that is a unital $*$-algebra $\NC$
equipped with a $*$-linear form $\tr:\, \NC\to\CC$
such that $\tr(1)=1$ and for all
$x,y\in\NC$, $\tr(xy)=\tr(yx)$ (trace property),
$\tr(x^*x)\ge0$ (non-negativity), and $\tr(x^*x)=0$ if and only
if $x=0$ (faithfulness).
Elements of $\NC$ are called
\textit{noncommutative random variables}, and the
\textit{distribution} of $(x_1,\ldots,x_n)\in\NC^n$ is the family
of \textit{moments}
$(\tr(x_{i_1}\cdots x_{i_r})\colon r\ge1,1\le i_1,\ldots,i_r\le n)$. 
Finally, we recall the notion of \textit{freeness}, which is the
counterpart of classical independence:
we say that (unital)
sub-algebras $\NC_1,\NC_2,\ldots\subseteq\NC$ are
\textit{free} if ``any alternating product of centered elements
is centered'', i.e., $\tr(x_1\cdots x_r)=0$ whenever
$\tr(x_j)=0$ for all $1\le j\le r$ and $x_j\in\NC_{i_j}$ with
$i_1\neq i_2,i_2\neq i_3,\ldots,i_{r-1}\neq i_r$.
In this vein, noncommutative random variables are free if the
respective sub-$*$%
-algebras they span are free.

The inner product
\[\langle x,y\rangle\ceq\tr(x^*y),\qquad(x,y)\in\NC^2,\]
confers on~$\NC$ a natural~$\Le^2$-structure with
norm $\|x\|_2\ceq\sqrt{\langle x,x}\rangle$, so we may
w.l.o.g.\ consider instead its Hilbert space completion
$\Le^2(\NC)$.
More generally, $\Le^2(\NC')$ will denote the closure
in~$\Le^2(\NC)$ of the sub-algebra $\NC'\subseteq\NC$.
If $\NC'=\CC\langle x_1,\ldots,x_n
\rangle$
is the sub-$*$%
-algebra of noncommutative polynomials in
$x_1,\ldots,x_n\in\NC$
\emph{and}
$x_1^*,\ldots,x_n^*\in \NC$%
, we refer to~$\Le^2(\NC')$ as
$\Le^2(x_1,\ldots,x_n)$.


Let
$\cov(x,y)\ceq\langle x-\tr(x)\cdot1,y-\tr(y)\cdot1\rangle$ and
$\std(x)^2\ceq\cov(x,x)$. Just like in classical probability,
we can define the Pearson correlation coefficient between~$x$
and~$y$ by
\[\rho(x,y)\ceq\frac{\cov(x,y)}{\std(x)\std(y)}.\]
Note that we have $\lvert\rho(x,y)\rvert\le1$ by the Cauchy--Schwarz
inequality. Given
$\NC_1,\NC_2\subseteq\NC$ two subspaces, we call
\[R(\NC_1,\NC_2)\ceq\!\sup_{\NC_1\times\NC_2}\rho\]
the \textit{maximal correlation coefficient} between~$\NC_1$
and~$\NC_2$.
For $\NC_1=\Le^2(x)$ and $\NC_2=\Le^2(y)$ we simply write~%
$R(x,y)$, and we call it the maximum correlation
between the noncommutative random variables~$x$ and~$y$.

We are now ready to state our main result.  
\begin{theorem}\label{thm:maxcorr}
	Let $(x_i)_{i\in \NN}$ be a sequence of free, non-zero, identically
	distributed, noncommutative random variables, and let
	$s_n\ceq x_1+\cdots+x_n$. Then for any $m\le n$, we have 
	$$
	R(s_n,s_m)=\sqrt{\frac{m}{n}}.
	$$
\end{theorem}

As in the classical setting, the interesting feature of the above statement is its universality as it 
 holds regardless of the distribution of the noncommutative random variables. 
A possible way to prove the above statement consists of using the microstate approach 
by approximating the law of each noncommutative random variable by that of random matrices. One then 
exploits the multidimensional version of the classical maximal correlation inequality to apply it 
for the corresponding random matrices (seen as vectors) before passing to the limit and deducing the above theorem. 
The drawback of this approach is that it won't allow the extension of Theorem~\ref{thm:maxcorr} to 
the multidimensional case. Indeed, by the refutation of Connes embedding problem \cite{JNVWY}, 
there are noncommutative random variables whose joint moments cannot be approximated well
by moments of matrices. This makes it impossible to use the mentioned approach to prove a multidimensional 
version of the above theorem. 
On the contrast, our proof of Theorem~\ref{thm:maxcorr}, which is carried in Section~\ref{sec:proof}, adapts 
the approach of~\cite{Dembo01}  to a noncommutative setting and is readily extendable to the multidimensional setting.   

A celebrated result of Artstein et al \cite{Artstein04} provided a solution to Shannon's problem regarding 
the monotonicity of entropy in the classical central limit theorem. 
In the context of free probability, the concept of \textit{free entropy} and \textit{information} was developped by Voiculescu 
in a series of papers (see for example \cite{Voiculescu02}). 
Two approaches were given for the definition of free entropy, referred to as \textit{microstates} and \textit{non-microstates} and denoted by~$\chi$ and~$\chi^*$ respectively (see~\cite[Chapters~7 and~8]{Mingo17}). These two coincide in the one-dimensional setting, in which case  
the free entropy of a compactly supported probability measure $\mu$ is given by 
$$
\chi(\mu)=\chi^*(\mu):=  \iint_{\RR^2} \log{\lvert a-b\rvert}\, \mu(\dd a)\mu(\dd b)+ \frac34+\frac12 \log (2\pi).
$$
It is not known whether $\chi$ and $\chi^*$ coincide in the multidimensional setting. Our proof of the result presented in the sequel 
extends to the multi-dimensional case for $\chi^*$ only. 
Given a noncommutative probability space $(\NC,\tr)$ and a self-adjoint element $z\in M$, we define~$\chi^*(z)$ as $\chi^*(\mu_z)$ where 
$\mu_z$ denotes the distribution of~$z$, i.e., the probability measure characterized by $\int p\,\dd\mu_z= \tr\big(p(z)\big)$ for all polynomials $p\in \CC[X]$. 

In \cite{Shlyakhtenko07}, Shlyakhtenko proved the monotonicity of the free entropy in the free central limit theorem providing an analogue of the 
result of \cite{Artstein04} in the noncommutative setting. As an application of our maximal correlation estimate, we recover the monotonicity property 
which we state in the next corollary. 
\begin{corollary}\label{cor:monotonicity}
Given $(x_i)_{i\in \NN}$ a sequence of free, identically 
distributed, self-adjoint random variables in $(\NC,\tr)$, one has 
\begin{equation}\label{eq: monotonicity}
\chi^*\left(\frac{s_m}{\sqrt m}\right)\leq \chi^*\left(\frac{s_n}{\sqrt n}\right),
\end{equation}
for every integers $m\leq n$. 
\end{corollary}

As in the classical setting, the monotonicity of the entropy follows from that of the Fisher information. 
In Section~\ref{sec:entropy}, we prove the latter as a consequence of Theorem~\ref{thm:maxcorr}. 
The idea of using the maximal correlation inequality in this context goes back to Courtade~\cite{Courtade16} who used the result of 
Dembo, Kagan, and Shepp~\cite{Dembo01} to provide an alternative proof of the monotonicity of entropy in the classical setting. 

Formulated alternatively, the above corollary states that given a compactly supported probability measure $\mu$ and positive integers $m\le n$, one has \begin{equation}
  \chi^*({m^{-1/2}}_*\mu^{\boxplus m})\le \chi^*({n^{-1/2}}_*
\mu^{\boxplus n}),
\label{eq:entconv}
\end{equation}
where~$\boxplus$ denotes the free convolution
operation (so~$\mu^{\boxplus n}$ is the
distribution of the sum of~$n$ free copies of an element~$x$ with distribution $\mu$),
and ${\alpha}_*$ is the pushforward operation by the
dilation $t\mapsto\alpha t$.
As a matter of fact, it is possible to make sense
of~$\mu^{\boxplus t}$ for all \emph{real} $t\ge1$
(see~\cite{Nica96}). Very recently, Shlyakhtenko and Tao~\cite{Tao20} extended \eqref{eq:entconv} to real-valued exponents while providing two different proofs. 
It would be interesting to see if the argument in this paper based on maximal correlation could be extended to cover non-integer exponents.

Another consequence of Theorem~\ref{thm:maxcorr} concerns the monotonicity of the \textit{free Stein discrepancy} along the free central limit theorem. 
Stein discrepancy measures in some sense how far is a probability measure from another one characterized by some integration by parts formula. 
Using the classical maximal correlation inequality, it was proven by Courtade, Fathi and  Pananjady \cite{CFP19} that the Stein discrepancy (relative to the standard Gaussian measure) 
is non-increasing in the central limit theorem. 
The notion of free Stein discrepancy relative to a semicircular law was introduced by Fathi and Nelson \cite{FN17}. 
Recall that a standard semicircular variable  $S$ is a self-adjoint element of $(\NC,\tr)$ who distribution has density $\frac{1}{2\pi}\sqrt{4-t^2}\, \mathbf{1}_{[-2,2]}(t)$. 
Analogously to the normal distribution, a standard semicircular variable $S\in \NC$ is characterized by the following integration by parts formula stating that 
$$
\langle S, P(S)\rangle = \langle 1\otimes 1, \partial P (S)\rangle,
$$
for every polynomial $P$. Here, $\partial$ denotes the noncommutative derivative and the right hand side dot product refers to the dot product in the Hilbert space 
$\Le^2(\NC)\otimes \Le^2(\NC)$ (see \cite{Mingo17}). 
Following \cite{FN17}, a \textit{free Stein kernel} of $x\in \NC$ is an element $K\in \Le^2(\NC)\otimes \Le^2(\NC)$ such that 
$$
\langle x, P(x)\rangle = \langle K, \partial P (x)\rangle,
$$
for every polynomial $P$. It was shown by C\'ebron, Fathi and Mai \cite{CFM20} that free Stein kernel always exist if $\tr(x)=0$. 
The free Stein discrepancy of $x$ relative to $S$ is then defined as 
$$
\Sigma^*(x\mid S)= \inf_{K} \Vert K-1\otimes 1\Vert_{\Le^2(\NC)\otimes \Le^2(\NC)},
$$
where the infimum is taken over all free Stein kernels $K$ of $x$. We should note that \cite{FN17} introduced the notion of 
free Stein kernel/discrepancy relative to a general potential while we will only be dealing with the particular case of the potential $t^2/2$ 
leading to the notions stated above. 

As an application of our maximal correlation inequality, we obtain the following corollary extending 
the aforementioned monotonicity of Stein discrepancy obtained in \cite{CFP19} to the free setting. 

\begin{corollary}\label{cor:monotonicity-stein}
Given $(x_i)_{i\in \NN}$ a sequence of free, centered, identically 
distributed, self-adjoint random variables in $(\NC,\tr)$ with unit norm, one has 
\begin{equation}\label{eq: monotonicity-stein}
\Sigma^*\left(\frac{s_n}{\sqrt n}\mid S\right)\leq \sqrt{\frac{m}{n}}\; \Sigma^*\left(\frac{s_m}{\sqrt m}\mid S\right),
\end{equation}
for every integers $m\leq n$. 
\end{corollary}

Note that taking $m=1$ in the above corollary, we obtain that  $\Sigma^*\left(\frac{s_n}{\sqrt n}\mid S\right)$ decays 
faster than $C/\sqrt{n}$ for some constant $C$ recovering a result of \cite{CFM20}. 
Similarly to the previous results of this paper, the proof of the above corollary works verbatim in the multidimensional setting. 

\subsection*{Aknowledgement:} 
The authors would like to thank the anonymous referees for their numerous 
generous suggestions which greatly improved the manuscript. For instance, the monotonicity of free Stein discrepancy and its proof was suggested by one of the referees.
 The authors are grateful to Roland Speicher for helpful comments and for bringing to their attention the recent preprint \cite{Tao20}. 
The second named author is thankful to Marwa Banna for helpful discussions.

\section{Proof of Theorem~\ref*{thm:maxcorr}}\label{sec:proof}
Let us fix $x_1,x_2,\ldots\in\NC$. For every finite set
$I\subset\NN$,
we denote 
$$\Le^2_I\ceq\Le^2(x_i\colon i\in I) \quad\text{ and }\quad 
\proj_I(z)\ceq\proj_{\Le^2_I}(z)
$$ 
the orthogonal
projection of $z\in\NC$ onto~$\Le^2_I$, which is
nothing else but the \textit{conditional expectation}~%
of~$z$ given~$\Le^2_I$:
\[y=\proj_I(z)
\iff y\in\Le^2_I\enspace\text{and}\enspace\forall x\in\Le^2_I,\ 
\tr(xy)=\tr(xz).\]
In particular $\proj_I(z)=z$ if $z\in\Le^2_I$ and by the trace property 
$$
\proj_I(xzy)=x\proj_I(z)y\quad \text{ for all
$x,y\in\Le^2_I$. }
$$ 
Note that $\proj_\es(z)=\tr(z)\cdot1$ and $\proj_J\circ\proj_I=\proj_J$
for every $J\subseteq I$ (tower property).
When freeness is further involved, we can say a bit more. The following lemma appears in different forms in the 
literature (\cite{Biane}, \cite[Section~2.5]{Mingo17}), we include its proof 
for completeness. 
\begin{lemma}\label{lem:proj}
  Let $I,J\subset\NN$ be finite sets and suppose
  $(x_k\colon k\in I\cup J)$
  is free. Then:
  \begin{itemize}
  	\item[(i)]\label{lem:proj.1}
  	  if~$z$ is a (noncommutative) polynomial in variables in
  	  $\CC\langle x_j\colon j\in J\rangle$, then
  	  $\proj_I(z)$ is a polynomial in only those variables
  	  that are actually in
  	  $\CC\langle x_k\colon k\in I\cap J\rangle$;
  	\item[(ii)]\label{lem:proj.2}
  	  the projections commute:
  	  $\proj_I\circ\proj_J=\proj_{I\cap J}$.
  \end{itemize}
\end{lemma}
\begin{proof}
  \begin{itemize}
  \item[(i)] By linearity of~$\proj_I$, we may suppose without loss of generality that $z=a_1\cdots a_r$ with
  $a_j\in\CC\langle x_{i_j}\rangle$ and
  consecutively distinct indices $i_1,\ldots,i_r\in J$.
  From the moment-cumulant formula
  \cite[Definition~9.2.7]{Mingo17}
  w.r.t.\ the conditional expectation~$\proj_I$, we can write
  \end{itemize}
\[\proj_I(z)\ceq\!\!\!\sum_{\pi\in\mathrm{NC}(r)}\!\!\!\!\!
  \kappa^I_\pi(a_1,\ldots,a_r),\]
  where the summation ranges over all non-crossing partitions
  $\pi\in\mathrm{NC}(r)$ of $\{1,\ldots,r\}$ and the
  $\kappa^I_\pi$ are nestings (consistently with the blocks
  of~$\pi$) of the free (conditional) cumulants
  $\kappa^I_n\colon\NC^n\to\Le^2_I$ (which
  can be defined inductively). For instance, if
  \begin{gather*}
  	I=\{1,4\},\enspace J=\{1,2,3,4\},\enspace
  	z=x_1^4\cdot x_2^3\cdot x_4\cdot x_1^*\cdot x_2^*\cdot x_3^2\cdot x_1x_1^*\cdot x_4^3\cdot x_3^5\cdot x_4\ (r=10),\\\text{and}\enspace
  	\pi=\!\!\Matching{10}{1/10, 2/5, 5/9, 3/4, 7/8}\!\!,
  \end{gather*}
  then
  \[\kappa^I_\pi(a_1,\ldots,a_r)=
  \kappa^I_2\biggl(x_1^4\cdot\kappa^I_3\Bigl(x_2^3
  \cdot\kappa^I_2\bigl(x_4,x_1^*\bigr),x_2^*\cdot
  \kappa^I_1\bigl(x_3^2\bigr)
  \cdot\kappa^I_2\bigl(x_1x_1^*,x_4^3\bigr),x_3^5\Bigr),x_4
  \biggr).\]
  We now invoke \cite[Theorem~3.6]{Nica02} with
  $F\ceq\tau\rvert_{\Le^2_I}$ and the sub-algebra
  $N\ceq\CC\langle x_j\colon j\in J\setminus I\rangle$ free from
  $B\ceq\Le^2_I$ over $D\ceq\CC\cdot1$
  to see that all
  conditional cumulants involving any
  variable in~$N$
  reduce to constants; e.g.,
  $\kappa^I_3(\ldots)=\tau\bigl(\kappa^I_2(x_4,x_1^*)\bigr)
  \tau\bigl(x_3^2\bigr)
  \tau\bigl(\kappa^I_2(x_1x_1^*,x_4^3)\bigr)
  \kappa_3(x_2^3,x_2^*,x_3^5)$
  in the above display,
  where~$\kappa_3$ is the (unconditionned) cumulant
  $\NC^3\to\CC$. This shows
  that~$\proj_I(z)$ is in fact a polynomial only in the subset of
  the variables $a_1,\ldots,a_r$ that belong to
  $\CC\langle x_k\colon k\in I\cap J\rangle$ (in the example,
  $\kappa^I_\pi(a_1,\ldots,a_r)
  =\tau\bigl(\kappa^I_2(x_4^*,x_1)\bigr)
  \tau\bigl(x_3^2\bigr)
  \tau\bigl(\kappa^I_2(x_1x_1^*,x_4^3)\bigr)
  \kappa_3(x_2^3,x_2^*,x_3^5)\,
  \kappa^I_2(x_1^4,x_4)$ is indeed a polynomial in
  the $\CC\langle x_1,x_4\rangle$-variables
  $a_1=x_1^4,\,a_3=a_{10}=x_4,\,a_4=x_1^*,\,a_7=x_1x_1^*,\,a_8=x_4^3$).

  \noindent
 \begin{itemize}
 	\item[(ii)]
  Clearly
  $\tr\bigl(x\proj_I\circ\proj_J(z)\bigr)
  =\tr\bigl(\proj_I(\proj_J(xz))\bigr)
  =\tr(xz)$
  for all $x\in\Le^2_{I\cap J}$, so we only need to check that
  $\proj_I\bigl(\proj_J(z)\bigr)\in\Le^2_{I\cap J}$.
  By a closure
  argument, we may suppose that
  $\proj_J(z)$ belongs to $\CC\langle x_j\colon j\in J
  \rangle$.
  In this case
  $\proj_I\bigl(\proj_J(z)\bigr)\in\CC\langle x_k\colon k\in I\cap J\rangle\subseteq\Le^2_{I\cap J}$
  by the previous point.\qedhere
  \end{itemize}
\end{proof}

We now set, for every $z\in\NC$,
\begin{equation}
  z_I\ceq\sum_{J\subseteq I}(-1)^{|I|-|J|}\proj_J(z)
\in\Le^2_I,\label{eq:defzpart}
\end{equation}
where~$|\cdot|$ is the cardinal notation. The following
decomposition will play a crucial role in the proof of Theorem~\ref{thm:maxcorr}. 
\begin{lemma}[Efron--Stein decomposition]\label{lem:efronstein}
	For every finite set $I\subset\NN$,
	\begin{equation*}
    	\proj_I(z)=\sum_{J\subseteq I}z_J.
	\end{equation*}
\end{lemma}
\begin{proof}
	We repeat in a compact way the argument of Efron and
	Stein~\cite{Efron81}:
	\begin{align*}
		\sum_{J\subseteq I}z_J
		&=\sum_{J\subseteq I}\sum_{K\subseteq J}
		(-1)^{|J|-|K|}\proj_K(z)\\
		&=\sum_{K\subseteq I}\underbrace{\left(\sum_{K\subseteq J\subseteq I}\!\!\!
			(-1)^{|J|-|K|}\right)}_{\text{$=(1-1)^{|I\setminus K|}$}}
		\proj_K(z)\\
		&=\proj_I(z).\qedhere
	\end{align*}
\end{proof}

The elements~$z_J$ will be orthogonal thanks to this direct
consequence of Lemma~\ref{lem:proj}:
\begin{lemma}\label{lem:orthogonality}
	Suppose that $x_1,x_2,\ldots$ are free, and
	let $I,J\subset\NN$ be finite sets such that $I\setminus J \neq\es$. Then $\proj_J(z_I)=0$ for every
	$z\in\NC$. In particular, $z_I$ is orthogonal
	to $z_J\in\Le^2_J$.
\end{lemma}
\begin{proof}
  Apply Lemma~\ref{lem:proj.2} 
  and gather the subsets
  $K\subseteq I$ that have same intersection $L\ceq J\cap K$
  with~$J$:
  \begin{align*}
  	\proj_J(z_I)
  	&=\sum_{K\subseteq I}(-1)^{|I|-|K|}\proj_J\circ \proj_K(z)\\
  	&=\!\!\!\sum_{L\subseteq I\cap J}\!\!\!(-1)^{|I|-|L|}
  	\underbrace{\left(\sum_{K\subseteq I\setminus J}\!\!\!
  	(-1)^{|K|}\right)}_{=(1-1)^{|I\setminus J|}}
  \proj_L(z)\\&=0.\qedhere
  \end{align*}
\end{proof}

To prove Theorem~\ref{thm:maxcorr} we shall finally exploit the fact
that the partial sum $s_n\ceq x_1+\cdots+x_n$ is
\emph{symmetric} in $(x_1,\ldots,x_n)$. Our next proposition is
tailored for this purpose.
\begin{proposition}\label{pro:symmetry}
  Suppose that $x_1,\ldots,x_n$
  are free and identically distributed.
  \begin{enumerate}
  	\item For every symmetric polynomial $z=p(x_1,\ldots,x_n)$ in
  	  $x_1,\ldots,x_n$ and every $I\subseteq\{1,\ldots, n\}$,
  	  the pair
  	  $(z_I,z_I^*)$
  	  has the same distribution as
  	  $(z_{I'},z_{I'}^*)$ where $I'\ceq\{1,\ldots,|I|\}$.\\
      Consequently, if $\tr(z)=0$, then
      \[\bigl\|\proj_I(z)\bigr\|_2
         \le\sqrt{\frac{|I|}n}\,\|z\|_2.\]
    \item For every $m\le n$
     and every polynomial~$p$,
    \[\proj_{\Le^2(x_1,\ldots,x_m)}\bigl(p(s_n)\bigr)
      =\proj_{\Le^2(s_m)}\bigl(p(s_n)\bigr).\]
  \end{enumerate}
\end{proposition}
\begin{proof}
  \begin{enumerate}
  	\item Let $m=|I|$ and let~$\varsigma$ be a permutation of
  	$\{1,\ldots,n\}$ mapping~$I'$ to~$I$.
  	Since~$p$ is symmetric, we can write
  	$z=p(x_{\varsigma(1)},\ldots,x_{\varsigma(n)})$.
  	By Lemma~\ref{lem:proj.1}, we see from its
  	definition in~\eqref{eq:defzpart} that~%
  	$z_I$ is a polynomial $q(x_{\varsigma(1)},
  	\ldots,x_{\varsigma(m)})$ in the 
  	$x_i,\,i\in I$. Now the pairs
  	$(z_I,z_I^*)=(q(x_{\varsigma(1)},\ldots,x_{\varsigma(m)}),
  	q(x_{\varsigma(1)},\ldots,x_{\varsigma(m)})^*)$ and
  	$(z_{I'},z_{I'})=(q(x_1,\ldots,x_m),q(x_1,\ldots,x_m)^*)$
  	have the same moments because $x_1,\ldots,x_n$ are free
  	and identically distributed. In particular
  	$\|z_I\|_2=\|z_{I'}\|_2$, and the stated inequality
  	becomes clear by combining
  	this with Lemmas~\ref{lem:efronstein} and \ref{lem:orthogonality}:
  	\begin{align*}
  	  \bigl\|{\proj_I(z)}\bigr\|_2^2
  	  &=\sum_{J\subseteq I}\|z_J\|_2^2\\
  	  &=\sum_{k=1}^m\binom mk\|z_{\{1,\ldots,k\}}\|_2^2\\
  	  &\le\frac mn\sum_{k=1}^n\binom nk\|z_{\{1,\ldots,k\}}\|_2^2\\[.4em]
  	  &=\frac mn\,\bigl\|{\proj_{\{1,\ldots,n\}}(z)}\bigr\|_2^2\\%
  	  [.4em]
  	  &=\frac mn\,\|z\|_2^2,
  	\end{align*}
    using that $z_\es=0$, that
    $\binom mk\le\frac mn\binom nk$
    for all $1\le k\le m\le n$,
    and that $z\in\Le^2(x_1,\ldots,x_n)$.
    \item We only need to check that
    $\proj_{\Le^2(x_1,\ldots,x_m)}(p(s_n))\in\Le^2(s_m)$ since
    \[\forall y\in\Le^2(s_m),\ 
      \tr\bigl(y\proj_{\Le^2(x_1,\ldots,x_m)}(p(s_n))\bigr)
      =\tr\bigl(\proj_{\Le^2(x_1,\ldots,x_m)}(yp(s_n))\bigr)
      =\tr\bigl(yp(s_n)\bigr)\]
    (as $\Le^2(s_m)\subset\Le^2(x_1,\ldots,x_m)$).
    But $p(s_n)$ is a polynomial in $s_m,x_{m+1},\ldots,
    x_n$ and freeness of $s_m,x_{m+1},\ldots,x_n$
    again implies by Lemma~\ref{lem:proj.1}
    that
    $\proj_{\Le^2(x_1,\ldots,x_m)}(p(s_n))\in\CC\langle s_m\rangle$.\qedhere
  \end{enumerate}
\end{proof}
\begin{proof}[Proof of Theorem~\ref*{thm:maxcorr}]
  The lower bound $R(s_n,s_m)\ge\sqrt{m/n}$ is straightforward
  since, by freeness,
  $\std(s_n)^2=n\std(x_1)^2$
  and $\cov(s_n,s_m)=\std(s_m)^2=m\std(x_1)^2$.
  For the upper bound, we must show that
  $\rho(z,z')\le\sqrt{m/n}$ for all $z\in\Le^2(s_n)$ and $z'\in\Le^2(s_m)$. W.l.o.g.,
  we may suppose that $\tau(z)=\tau(z')=0$ and, by another
  closure argument, that~$z$ is a polynomial in~$s_n$ (and thus
  a symmetric polynomial in~$x_1,\ldots,x_n$).
  Then by the Cauchy--Schwarz inequality
  and Proposition~\ref{pro:symmetry},
  \begin{align*}
\cov(z,z')&=\langle z,z'\rangle\\
&=\bigl\langle\proj_{\Le^2(s_m)}(z),z'\bigr\rangle\\
&\le\bigl\|\proj_{\Le^2(s_m)}(z)\bigr\|_2\,\|z'\|_2\\
&\le\sqrt{\frac mn}\,\|z\|_2\,
  \|z'\|_2\\
&=\sqrt{\frac mn}\std(z)\std(z'),
  \end{align*}
  and the proof is complete.
\end{proof}

\section{Monotonicity of the free entropy and free Fisher information}\label{sec:entropy}

The goal of this section is to prove Corollary~\ref{cor:monotonicity}. 
Let us start by noting that the free entropy and free Fisher information of a self-adjoint element $z\in \NC$ are related 
through
the integral formula (see~\cite[Chapter~8]{Mingo17})
\begin{equation}\label{eq: integral representation}
\chi^*(z)=\frac12\int_0^\infty\left(\frac1{1+t}-\Phi\!\left(z+\sqrt t\,x\right)\!
\right)\,\dd t+\frac12\log(2\pi\operatorname{e}),
\end{equation}
where~$x$ is a standard semi-circular variable free from~$z$, and $\Phi$ denotes the free Fisher information. 
After \cite{Voiculescu98} (see also \cite[Chapter~8]{Mingo17}), the free
Fisher information of a noncommutative, self-adjoint
random variable $z\in\NC$ is defined as
$\Phi(z)\ceq\|\xi\|_2^2$
where the so called \textit{conjugate variable}
$\xi\ceq\xi(z)$ is any element of~$\Le^2(z)$ such that,
for every integer $r\ge0$,
\begin{equation}
  \tr\bigl(\xi z^r\bigr)
  =\sum_{k=0}^{r-1}\tr\bigl(z^k\bigl)\tr\bigl(z^{r-1-k}\bigr).%
  \label{eq:fisher}
\end{equation}
(If such a~$\xi$ does not exist, we set~$\Phi(z)\ceq\infty$.) We note from~\eqref{eq:fisher} that
$\tau(\xi(z))=0$ and the homogeneity property
$\Phi(\alpha z)=\alpha^{-2}\,\Phi(z),\,\alpha>0$.

In the next Corollary, we show how the monotonicity of the free Fisher information follows easily from Theorem~\ref{thm:maxcorr}. 
\begin{corollary}\label{cor:fisher}
  Let $(x_i)_{i\in \NN}$ be a sequence of free, identically 
distributed, self-adjoint random variables in $(\NC,\tr)$ and denote $s_k:= x_1+\ldots +x_k$ for every positive integer $k$.  Then for all positive integers $m\le n$, we have 
  \[\Phi\left(\frac{s_n}{\sqrt n}\right)
  \le\Phi\left(\frac{s_m}{\sqrt m}\right).\]
\end{corollary}
\begin{proof}
  Assume the existence of~$\xi(s_1)$, as otherwise
  $\Phi(s_1)=\infty$ and there is nothing to prove.
  According to~\cite[p.\ 206]{Mingo17},
  the free sum $s_n=s_m+(s_n-s_m)$ admits
  $\xi(s_n)=\proj_{\Le^2(s_n)}(\xi(s_m))$ as conjugate variable.
  Therefore, by Theorem~\ref{thm:maxcorr},
  \begin{align*}
  	\Phi(s_n)=\bigl\|\proj_{\Le^2(s_n)}\bigl(\xi(s_m)\bigr)\bigr\|_2^2
  	&=\cov\left(\proj_{\Le^2(s_n)}\bigl(\xi(s_m)\bigr),\xi(s_m)\right)\\
  	&\le\sqrt{\frac mn}\,\bigl\|\proj_{\Le^2(s_n)}\bigl(\xi(s_m)\bigr)
  	  \bigr\|_2\,\bigl\|\xi(s_m)\bigr\|_2,
  \end{align*}
i.e., $\Phi(s_n)\le\frac mn\Phi(s_m)$. We conclude
by the homogeneity property.
\end{proof}
In view of~\eqref{eq: integral representation} and the
divisibility of the semicircular distribution w.r.t.\ the free
convolution, the above
corollary readily implies \eqref{eq: monotonicity},
thus proving Corollary~\ref{cor:monotonicity}.

\section{Monotonicity of the free Stein discrepancy}\label{sec:stein}

The goal of this section is to provide a proof of Corollary~\ref{cor:monotonicity-stein}. 
Let us fix  $m\leq n$ and $x_1,x_2,\ldots$ a sequence of free, centered, identically 
distributed, self-adjoint random variables in $(\NC,\tr)$ with unit norm. 
Let us record the following consequence of Theorem~\ref{thm:maxcorr} which will be used in the sequel. 

\begin{lemma}\label{lem:tensor-proj-norm}
Let $K\in \Le^2(s_m)\otimes \Le^2(s_m)$ and suppose that $\langle K, 1\otimes 1\rangle= 1$. Then 
$$
\Vert \proj_{\Le^2(s_n)\otimes \Le^2(s_n)}(K-1\otimes 1) \Vert_{\Le^2(s_n)\otimes \Le^2(s_n)}\leq 
\sqrt{\frac{m}{n}}\, \Vert K-1\otimes 1\Vert_{\Le^2(s_m)\otimes \Le^2(s_m)}.
$$
\end{lemma}
\begin{proof}
We may suppose without loss of generality that $\tr(b)=0$. 
We start by writing 
\begin{align*}
\Vert \proj_{\Le^2(s_n)\otimes \Le^2(s_n)} (a\otimes b)\Vert_{\Le^2(s_n)\otimes \Le^2(s_n)}
&=\Vert \proj_{\Le^2(s_n)}(a)\Vert_{\Le^2(s_n)}\, \Vert \proj_{\Le^2(s_n)}(b)\Vert_{\Le^2(s_n)}\\
&\leq \Vert a\Vert_{\Le^2(s_n)}\, \Vert \proj_{\Le^2(s_n)}(b)\Vert_{\Le^2(s_n)}.
\end{align*}
Now using Theorem~\ref{thm:maxcorr}, we have 
$$
\Vert \proj_{\Le^2(s_n)}(b)\Vert_{\Le^2(s_n)}^2=\langle \proj_{\Le^2(s_n)}(b), b\rangle \leq \sqrt{\frac{m}{n}}\, \Vert \proj_{\Le^2(s_n)}(b)\Vert_{\Le^2(s_n)} \, \Vert b\Vert_{\Le^2(s_n)},
$$
which yields to 
$$
\Vert \proj_{\Le^2(s_n)}(b)\Vert_{\Le^2(s_n)}\leq  \sqrt{\frac{m}{n}}\, \Vert b\Vert_{\Le^2(s_n)}.
$$
Putting the above together, we get 
$$
\Vert \proj_{\Le^2(s_n)\otimes \Le^2(s_n)} (a\otimes b)\Vert_{\Le^2(s_n)\otimes \Le^2(s_n)}\leq \sqrt{\frac{m}{n}}\, \Vert a\Vert_{\Le^2(s_n)}\,\Vert b\Vert_{\Le^2(s_n)},
$$
and finish the proof. 
\end{proof}

Given $K\in \Le^2(s_m)\otimes \Le^2(s_m)$ a free Stein kernel of $\frac{s_m}{\sqrt{m}}$, we have by the definition of $K$ (and homogeneity) that 
$$
\langle \frac{s_m}{\sqrt{m}}, P(s_m)\rangle =\sqrt{m}\, \langle K, \partial P(s_m)\rangle,
$$
for every polynomial $P$. 
Since $s_m$ and $s_n-s_m$ are free, then using the intertwining relation between $\partial$ 
and the conditional expectation \cite{Voiculescu00}, we can write 
$$
\langle \frac{s_m}{\sqrt{m}}, P(s_n)\rangle =\sqrt{m}\, \langle K, \partial P(s_n)\rangle.
$$
The above relation appears in \cite{Cebron} (before Lemma~2.5 there). 
Using linearity and that the $x_i$'s are identically distributed, we get 
$$
\langle s_m, P(s_n)\rangle= m\, \langle x_1, P(s_n)\rangle= \frac{m}{n}\, \langle s_n, P(s_n)\rangle.
$$
Putting the above relations together, we deduce that 
$$
\langle \frac{s_n}{\sqrt{n}}, P(s_n)\rangle= \sqrt{n}\, \langle K, \partial P(s_n)\rangle,
$$
for every polynomial $P$. This implies that $\proj_{\Le^2(s_n)\otimes \Le^2(s_n)}(K)$ is a free Stein kernel of 
$\frac{s_n}{\sqrt{n}}$. Thus, we have 
$$
\Sigma^*\left(\frac{s_n}{\sqrt n}\mid S\right)\leq \inf_K  \Vert \proj_{\Le^2(s_n)\otimes \Le^2(s_n)}(K-1\otimes 1)\Vert_{\Le^2(s_n)\otimes \Le^2(s_n)},
$$
where the infimum is taken over all free Stein kernels $K\in \Le^2(s_m)\otimes \Le^2(s_m)$ of $s_m/\sqrt{m}$. 
Noting that $\langle K-1\otimes 1,1\otimes 1\rangle =0$, then using Lemma~\ref{lem:tensor-proj-norm} we deduce 
$$
\Sigma^*\left(\frac{s_n}{\sqrt n}\mid S\right)\leq \sqrt{\frac{m}{n}}\, \inf_K  \Vert K-1\otimes 1\Vert_{\Le^2(s_n)\otimes \Le^2(s_n)},
$$
where the infimum is again taken over all free Stein kernels $K\in \Le^2(s_m)\otimes \Le^2(s_m)$ of $s_m/\sqrt{m}$. 
Corollary~\ref{cor:monotonicity-stein} then follows after noting that the projection on $\Le^2(s_m)\otimes \Le^2(s_m)$ of every free Stein 
kernel of $s_m/\sqrt{m}$ is also a free Stein kernel.

\noindent {\small Benjamin Dadoun,}\\
{\small Mathematics, Division of Science, New York University Abu Dhabi, UAE}\\
{\small \it E-mail: benjamin.dadoun@gmail.com}

\bigskip

\noindent {\small Pierre Youssef,}\\
{\small Mathematics, Division of Science, New York University Abu Dhabi, UAE}\\
{\small \it E-mail: yp27@nyu.edu}

\begin{thebibliography}{10}

\bibitem{Artstein04}
Shiri Artstein, Keith~M. Ball, Franck Barthe, and Assaf Naor, \emph{Solution of
  {S}hannon's problem on the monotonicity of entropy}, J. Amer. Math. Soc.
  \textbf{17} (2004), no.~4, 975--982. \MR{2083473}

	\bibitem{Bercovici93}
	Hari Bercovici and Dan Voiculescu, \emph{Free convolution of measures with
		unbounded support}, Indiana Univ. Math. J. \textbf{42} (1993), no.~3,
	733--773. \MR{1254116}

\bibitem{Biane}
Philippe Biane, \emph{Process with free increments}, Math. Z.227(1), 143--174 (1998).

\bibitem{CFM20}
Guillaume C\'ebron, Max Fathi, and Tobias Mai, \emph{A note on existence of free Stein kernels}, Proc. Amer. Math. Soc., 148 (2020), 1583--1594.

\bibitem{Cebron}
Guillaume C\'ebron, \emph{A quantitative fourth moment theorem in free probability theory}, Advances in
Mathematics, 380 (2021), 107579.

\bibitem{Courtade16}
Thomas~A. Courtade, \emph{Monotonicity of entropy and {F}isher information: a
  quick proof via maximal correlation}, Commun. Inf. Syst. \textbf{16} (2016),
  no.~2, 111--115. \MR{3638565}

\bibitem{CFP19}
Thomas A. Courtade, Max Fathi, and Ashwin Pananjady, \emph{Existence of Stein kernels under a spectral
gap, and discrepancy bounds},  Ann. Inst. H. Poincaré Probab. Statist., 55, 2 (2019), 777--790.

\bibitem{Dembo01}
Amir Dembo, Abram Kagan, and Lawrence~A. Shepp, \emph{Remarks on the maximum
  correlation coefficient}, Bernoulli \textbf{7} (2001), no.~2, 343--350.
  \MR{1828509}

\bibitem{Efron81}
B.~Efron and C.~Stein, \emph{The jackknife estimate of variance}, Ann. Statist.
  \textbf{9} (1981), no.~3, 586--596. \MR{615434}


\bibitem{FN17}
Max Fathi and Brent Nelson, \emph{Free Stein kernels and an improvement of the free logarithmic Sobolev
inequality},  Advances in Mathematics, 317 (2017), 193--223.

\bibitem{Hiai00}
Fumio Hiai and D\'{e}nes Petz, \emph{The semicircle law, free random variables
  and entropy}, Mathematical Surveys and Monographs, vol.~77, American
  Mathematical Society, Providence, RI, 2000. \MR{1746976}

\bibitem{JNVWY}
Zhengfeng Ji, Anand Natarajan, Thomas Vidick, John Wright and Henry Yuen, \emph{MIP*= RE}, arXiv preprint arXiv:2001.04383 (2020).

\bibitem{Mingo17}
James~A. Mingo and Roland Speicher, \emph{Free probability and random
  matrices}, Fields Institute Monographs, vol.~35, Springer, New York; Fields
  Institute for Research in Mathematical Sciences, Toronto, ON, 2017.
  \MR{3585560}

\bibitem{Nica96}
	Alexandru Nica and Roland Speicher, \emph{On the multiplication of free
		{$N$}-tuples of noncommutative random variables}, Amer. J. Math. \textbf{118}
	(1996), no.~4, 799--837. \MR{1400060}

\bibitem{Nica02}
Alexandru Nica, Dimitri Shlyakhtenko, and Roland Speicher,
  \emph{Operator-valued distributions. {I}. {C}haracterizations of freeness},
  Int. Math. Res. Not. (2002), no.~29, 1509--1538. \MR{1907203}


	
\bibitem{Shlyakhtenko07}
Dimitri Shlyakhtenko, \emph{A free analogue of Shannon's problem on monotonicity of entropy}, Adv. Math.
208 (2007), no. 2, 824Ð833. 

	\bibitem{Tao20}
	Dimitri Shlyakhtenko and Terence Tao.~With an~appendix~by David~Jekel,
	\emph{Fractional free convolution powers}, 2020. Available on arxiv:2009.01882. 





	\bibitem{Voiculescu94}
	Dan Voiculescu, \emph{The analogues of entropy and of {F}isher's information
		measure in free probability theory. {II}}, Invent. Math. \textbf{118} (1994),
	no.~3, 411--440. \MR{1296352}

	\bibitem{Voiculescu98}
	\bysame, \emph{The analogues of entropy and of {F}isher's information measure
		in free probability theory. {V}. {N}oncommutative {H}ilbert transforms},
	Invent. Math. \textbf{132} (1998), no.~1, 189--227. \MR{1618636}

\bibitem{Voiculescu00}
	\bysame, \emph{The coalgebra of the free difference quotient and free probability}, 
	 Int. Math. Res. Not. \textbf{2000} (2000), no. 2, 79--106.
	
	\bibitem{Voiculescu02}
	\bysame, \emph{Free entropy}, Bull. London Math. Soc. \textbf{34} (2002),
	no.~3, 257--278. \MR{1887698}
\end{thebibliography}
\end{document}